\documentclass[11pt]{article}

\usepackage{amsmath, amssymb, amsthm, color}
\usepackage{hyperref, url}
\usepackage{graphicx}               
\usepackage{color}

 \def\newblock{\ }%
\newtheorem{definition}{Definition}

\newtheorem{theorem}{Theorem}
\newtheorem{lemma}[theorem]{Lemma}
\newtheorem{corollary}[theorem]{Corollary}
\newtheorem{proposition}[theorem]{Proposition}

\newtheorem{example}{Example}

\newcommand{\ba}{\mbox{\boldmath $a$}}
\newcommand{\bb}{\mbox{\boldmath $b$}}
\newcommand{\bc}{\mbox{\boldmath $c$}}
\newcommand{\bd}{\mbox{\boldmath $d$}}
\newcommand{\be}{\mbox{\boldmath $e$}}
\newcommand{\bg}{\mbox{\boldmath $g$}}

\newcommand{\bx}{\mbox{\boldmath $x$}}
\newcommand{\by}{\mbox{\boldmath $y$}}

\newcommand{\bA}{\mbox{\boldmath $A$}}
\newcommand{\bB}{\mbox{\boldmath $B$}}

\newcommand{\bC}{\mbox{\boldmath $C$}}

\newcommand{\bI}{\mbox{\boldmath $I$}}
\newcommand{\bT}{\mbox{\boldmath $T$}}

\newcommand{\bU}{\mbox{\boldmath $U$}}

\newcommand{\bV}{\mbox{\boldmath $V$}}

\newcommand{\bX}{\mbox{\boldmath $X$}}
\newcommand{\bY}{\mbox{\boldmath $Y$}}

\newcommand{\bSigma}{\mbox{\boldmath $\Sigma$}}

\newcommand{\real}[1]{\mbox{$\mathbb{R}^{#1}$}}

\newcommand{\compl}[1]{\mbox{$\mathbb{C}^{#1}$}}

\newcommand{\complN}[1]{\mbox{$\compl{I_1 \times I_2 \times \dots \times I_{#1}}$}}
\newcommand{\complR}[1]{\mbox{$\compl{R_1 \times R_2 \times \dots \times R_{#1}}$}}

\newcommand{\rank}{\textnormal{rank}}

\newcommand{\Rmnum}[1]{\uppercase\expandafter{\romannumeral #1}} 
\usepackage[raggedright]{titlesec}
\titleformat{\chapter}{\centering\Huge\bfseries}{Chapter \Rmnum{\thechapter} }{1em}{} 
\newcommand{\df}{\em } 

\usepackage{mathrsfs} 
\usepackage{enumerate} 
\graphicspath{{figures/}}  

\newcommand{\tensorE}{\mbox{\boldmath $\mathcal{E}$}}

\newcommand{\tensorG}{\mbox{\boldmath $\mathcal{G}$}}
\newcommand{\tensorI}{\mbox{\boldmath $\mathcal{I}$}}
\newcommand{\tensorT}{\mbox{\boldmath $\mathcal{T}$}}

\newcommand{\tensorX}{\mbox{\boldmath $\mathcal{X}$}}
\newcommand{\tensorY}{\mbox{\boldmath $\mathcal{Y}$}}
\usepackage{algorithm}
\usepackage{algorithmicx}
\usepackage{multirow}
\usepackage{amsmath}
\usepackage{stmaryrd}
\usepackage{float}
\usepackage{graphicx}
\usepackage{subfigure}
\usepackage{xcolor}
\usepackage{algpseudocode}

\usepackage{color}

\begin{document}
\title{ Tensor and Its Tucker Core: the Invariance Relationships }

\author{
Bo JIANG
\thanks{Research Center for Management Science and Data Analytics, School of Information Management and Engineering, Shanghai University of Finance and Economics, Shanghai 200433, China. Email: isyebojiang@gmail.com. The research of this author was supported in part by National Natural Science Foundation of China (Grant 11401364) and Program for Innovative Research Team of Shanghai University of Finance and Economics.}
\and
Fan YANG
\thanks{Institute for Computational \& Mathematical Engineering, Stanford University, Stanford, CA 94305, USA. Email: fanfyang@stanford.edu.}
\and
Shuzhong ZHANG
\thanks{Department of Industrial and Systems Engineering, University of Minnesota, Minneapolis, MN 55455, USA. Email: zhangs@umn.edu. The research of this author was supported in part by National Science Foundation (Grant CMMI-1462408).}
}

\date{\today}
\maketitle

\begin{abstract}

In \cite{HillarLim2013}, Hillar and Lim famously demonstrated that ``multilinear (tensor) analogues of many efficiently computable problems in numerical linear
algebra are NP-hard''. Despite many recent advancements, the state-of-the-art methods for computing such `tensor analogues' still suffer severely from the curse of dimensionality. In this paper we show that the Tucker core of a tensor however, retains many properties of the original tensor, including the CP rank, the border rank, the tensor Schatten quasi norms, and the Z-eigenvalues.
{ When the core tensor is smaller than the original tensor}, this property leads to considerable computational advantages as confirmed by our numerical experiments.
In our analysis, we in fact work with a generalized Tucker-like decomposition that can accommodate any full column-rank factor matrices.

\vspace{0.5cm}

\noindent {\bf Keywords:} Tucker decomposition, CP decomposition, border rank, tensor Schatten quasi norm, tensor eigenvalues.

\noindent {\bf AMS subject classifications:} 15A69, 15A18, 15A03



\end{abstract}



\newpage

\section{Introduction}
A tensor is a multidimensional extension of matrices, which has recently attracted 
a surge of research attention due to its
wide applications in computer vision \cite{CRoMDUTRD}, psychometrics \cite{Harshman,CARROLL}, diffusion magnetic resonance imaging \cite{Ghosh2008,Bloy2008,HOPSDTI}, quantum entanglement problem \cite{TGMoMEatSVoaH}
and {tensor-structured numerical methods for multi-dimensional PDEs \cite{Khoromskij2006, Khoromskaia2010}}.
We refer the interested reader to {the surveys \cite{TDaA, Khoromskij2015} on these subjects}.

To study the spectral theory of tensors, various notions of tensor decompositions, eigenvalues and norms have been proposed. Unfortunately, unlike many of their matrix counter-parties,
most tensor problems are computationally intractable~\cite{HillarLim2013}. Therefore, the numerical algorithms that aim to globally solve those problems are often time-consuming.
For instance, the approach proposed in~\cite{AREoST} to compute all Z-eigenvalues of a tensor is based on the so-called SOS (sum of squares) approach, which leads to a series of Semidefinite Programs with fast-growing sizes. Thus, it is naturally desirable that the same computational task would be performed on a tensor with smaller size. 
In this paper
we establish 
that many of the aforementioned properties of a tensor carryover to 
its Tucker core, which is typically much smaller.   
To start off, let us introduce an extended notion of Tucker decomposition.
\begin{definition}   \label{Def:TensorDecomp}
Consider an $N$-way tensor $\tensorX \in \complN N$.
The equation
\begin{equation}   \label{Eq:DefTenforDecomp}
\tensorX = \tensorG \times_1 \bA^{(1)} \times_2 \cdots \times_N \bA^{(N)} \triangleq \llbracket \tensorG; \bA^{(1)} , \ldots, \bA^{(N)} \rrbracket,
\end{equation}
is called a
{\df size-$(J_1, \dots, J_N)$ Tucker decomposition} of $\tensorX$, and $\tensorG \in \compl{J_1 \times J_2 \times \cdots \times J_N}$ is called a {\df core tensor} associated with this decomposition, and $\bA^{(n)} \in \compl{I_n \times J_n}$ is the $n$-th {\df factor matrix} for $n=1, \dots, N$,
where $\times_n$ is the {\df mode-$n$ (matrix) product}.
Moreover, a Tucker decomposition is said to be {\df independent} if each of the factor matrix has full column rank;
a Tucker decomposition is said to be {\df orthonormal} if each of the factor matrix has orthonormal columns.
\end{definition}

Note that the conventional Tucker decomposition corresponds to the orthonormal Tucker decomposition, which is also known as the higher-order SVD (HOSVD) in the literature.
De Lathauwer, De Moor and Vandewalle~\cite{DeLathauwerMoorVandewalle2000} proposed an algorithm to compute such a decomposition. In a sequel, the same authors soon later proposed the higher-order orthogonal iteration (HOOI) in~\cite{DeLathauwerMoorVandewalle2000b} to accommodate for inexact Tucker decomposition.
In this paper, our analysis will be performed on the above-defined general independent Tucker decomposition unless specified otherwise.

The CP decomposition of a tensor is another important notion of tensor-decomposition, which leads to the definition of CP rank. The fact that the CP decomposition of core tensor is useful to decompose the original tensor itself
has already been observed (see \cite{BroAndersson1998} and Section 5.3 of \cite{TDaA}).
Consequently, the CP rank of a tensor equals that of its Tucker core follows from this observation.
Moreover, De Lathauwer et al.\ \cite{DeLathauwerMoorVandewalle2000} showed that the Frobenius norm remains invariant for a given tensor and its core. However, those results are scattered in the literature, and often they are implicit. 
In this paper, we aim to establish 
the equivalence between a tensor and its Tucker core in a systematic fashion, including the
CP rank and the Frobenius norm, 
and also other forms of
tensor ranks, Z-eigenvalues and tensor Schatten quasi norms. In addition to tensor decompositions, the study of tensor eigenvalues became popular after the seminal papers of
Qi~\cite{Qi2005} and Lim~\cite{SVaEoTAVA}. Furthermore, recently the tensor nuclear norm was used by Yuan and Zhang~\cite{YuanZhang2015} in
tensor completion to capture the low-rank structure;
the regression bound obtained in \cite{YuanZhang2015} is better than that induced by the 
mode-$n$ matricization. 
Those properties allow us to propose the following scheme to compute the rank, the norms and the eigenvalues of a tensor, as long as these properties are invariant between the tensor and its Tucker core.
As a first step, one computes Tucker decomposition of a given tensor.
Then, the computations are performed on the smaller Tucker core. Finally, the computed quantity is transformed back to the original tensor.
As we shall see later,
the size of the Tucker core may be no more than $2$ for some structured tensors, regardless the size of the original tensor.
The savings on the computational time
gained by following this scheme could be significant when the size of the core is considerably smaller compared to the original tensor,
which is the case for many specific instances encountered in our numerical experiments.

The remainder of this paper is organized as follows. In the next section, we introduce 
the tensor notations and operations that we shall use in this paper.  
Then we discuss the invariance of 
the tensor ranks, the norms and the eigenvalues in Sections~\ref{sec:rank}, \ref{sec:norm} and \ref{sec:eigenvalue}
respectively. Section \ref{sec:error} discusses the implications of the invariance properties, the advantages of size reduction, and 
an error estimation of the new computational scheme.
Finally, we
apply our scheme to compute all the Z-eigenvalues of symmetric 
tensors. Our numerical results show that our strategy leads to a significant reduction in computational time on a set of testing instances.

\section{Notations and preliminaries}

Throughout this paper, we use the boldface lowercase letters, the capital letters, and the calligraphic
letters to denote vectors, matrices, and tensors, respectively. For example, a vector (always a column vector unless otherwise stated) $\bx$, a matrix $\bA$,
and a tensor $\tensorX$. We use $\|\cdot \|$ to denote the Euclidean norm of the vectors. For a matrix $\bA$, $\sigma_{\max}(\bA)$ and $\sigma_{\min}(\bA)$ denote
the largest and smallest singular value of $\bA$,
while $\| \bA\|_2$ denotes its spectral norm:
$$
\| \bA\|_2 = \max_{\|\bx\|=1} \|\bA \bx\| = \sigma_{\max}(\bA).
$$
We use lowercase subscripts
to denote its components; e.g.\ $x_i$ is the $i$-th entry of
vector $\bx$, $a_{ij}$ is the $(i,j)$-th entry of matrix $\bA$, and $x_{i_1\cdots i_n}$ is the $(i_1,\cdots, i_n)$-th entry of $n$-th
order tensor $\tensorX$. Moreover, we use the superscripts with bracket to refer a sequence of variables; e.g., a sequence of $N$ matrices is denoted by $\bA^{(1)},\bA^{(2)}, \ldots, \bA^{(N)}$.

The {\df mode-$n$ (matrix) product} denoted by ``$\times_n $'' of a tensor $\tensorX \in \complN N$ and a matrix $\bA \in \compl{J \times I_n}$ results in a tensor of size $I_1 \times \cdots \times I_{n-1} \times J \times I_{n+1} \times \cdots \times I_{N}$
such that
\begin{equation*}
\left[ \tensorX \times_n \bA \right]_{i_1\cdots i_{n-1} j i_{n+1} \cdots i_N}=\sum_{i_n=1}^{I_n}{x_{i_1\cdots i_{n-1} i_n  i_{n+1}\cdots i_N} \cdot a_{ji_n}}.
\end{equation*}
In the meanwhile, for a given tensor $\tensorX \in \complN N$, the {\df mode-$n$ matricization} denoted by $\bX_{(n)}$ is a mapping from tensor to matrix. In particular, the $(i_1, i_2, \dots, i_N)$-th entry of $\tensorX$ corresponds to the $(i_n, j)$-th entry of $\bX_{(n)}$, where
$$
j=1+\sum_{k=1, k\neq n}^N{(i_k-1)J_k}, \ J_k=\prod_{m=1, m\neq n}^{k-1}{I_m}.
$$
Some properties relating the mode-$n$ product and mode-$n$ matricization are summarized in the following proposition, which will be used later. Interested reader is referred to
\cite{BaderKolda2006} for more information on 
tensor multiplications.
\begin{proposition}\label{Prp:NModeProd}
For any $N$-way tensor $\tensorX \in \complN N$ and matrix $\bU \in \compl{J\times I_n}$, their mode-$n$ product satisfies
$$
(\tensorX \times_n \bU)_{(n)}=\bU \bX_{(n)}.
$$
Moreover, for any given matrices $\bA \in \compl{J_m \times I_m}$, $\bB \in \compl{J_l \times I_l}$, if $m \neq l$ then we have
$$
(\tensorX \times_m \bA) \times_l \bB=(\tensorX \times_l \bB) \times_m \bA;
$$
if $m=l$ and suppose the matrix multiplication is compatible, then we have
$$
(\tensorX \times_m \bA) \times_m \bB=\tensorX \times_m (\bB  \bA).
$$
\end{proposition}
The {\df outer product} denoted by ``$\circ$'' of two tensors $\tensorX \in \complN {N_1}$ and $\tensorY \in \compl{I_{N_1+1} \times \dots \times I_{N_1+N_2}}$ is a tensor of size
$I_{1} \times \cdots \times I_{N_1 + N_2}$ such that
\begin{equation*}
\left[ \tensorX \circ \tensorY \right]_{i_1 \cdots i_{N_1+N_2}}=x_{i_1 \cdots i_{N_1}} \cdot y_{i_{N_1+1} \cdots i_{N_1+N_2}}.
\end{equation*}
In particular, $\tensorX$ is a rank-1 tensor if it can be written as an outer products of vectors; e.g.
\begin{equation*}\label{Eq:DefTensorProduct}
\tensorX=\ba^{(1)} \circ \ba^{(2)} \circ \dots \circ \ba^{(N)}.
\end{equation*}
The {\df inner product} of two tensors $\tensorX \in \complN {N}$ and $\tensorY \in \complN {N}$ is denoted by
\begin{equation*}
\langle \tensorX, \tensorY \rangle=\sum_{i_1=1}^{I_1}\dots\sum_{i_N=1}^{I_N}x_{i_1\cdots i_N} \cdot y_{i_1\cdots i_N}.
\end{equation*}
For any tensor $\tensorX\in \complN N$, the {\df Frobenius norm} of tensor $\tensorX$ is defined as
$$
\|\tensorX \|_F \triangleq \sqrt{\langle \tensorX, \tensorX \rangle}.
$$

More discussions on tensor operations can be found in~\cite{MOfHOD}.
The notion of tensor decomposition is central to the study of tensors.
Let us now formally introduce the so-called CP decomposition and the CP rank, where `CP' is a further abbreviation from CANDECOMP (canonical decomposition) by Carroll and
Chang~\cite{CARROLL} and PARAFAC (parallel factorization) by Harshman \cite{Harshman} in early 1970's. 

\begin{definition}   \label{Def:CP}
For an $N$-way tensor $\tensorX \in \complN N$, a
{\df CP decomposition}  {is a representation of} $\tensorX$ as a sum of $r$ rank-1 tensors,
\[
\tensorX = \sum_{t=1} ^r \ba^{(1,t)} \circ \cdots \circ \ba^{(N,t)},
\]
where $\ba^{(n,t)}\in \compl{I_n}$. The {\df CP rank} of $\tensorX$, denoted by $\rank_{CP}(\tensorX)$, is the minimum integer $r$ such that a size-$r$ CP decomposition is possible. 
\end{definition}
Another important notion of tensor decomposition is the so-called Tucker decomposition proposed by Tucker \cite{Tucker1966}, 
as introduced in Definition~\ref{Def:TensorDecomp} (though in a slightly more general format). Likewise, this decomposition also leads to another notion of tensor rank.
\begin{definition}   \label{Def:TuckerRank}
For an $N$-way tensor $\tensorX \in \complN N$, its {\df Tucker Rank}, denoted by $\rank_{T}(\tensorX)$, is an $N$-dimensional vector
\begin{equation}  
\rank_T (\tensorX)=(\rank(\bX_{(1)}), \dots, \rank(\bX_{(N)})),
\end{equation}
where $\bX_{(n)}$ is the mode-$n$ matricization of $\tensorX$ for $n=1, \dots, N$, and $rank(\cdot)$ denotes the regular matrix rank.
\end{definition}

In fact, many of the operations we have discussed about can be represented by Tucker decomposition.
For example, a rank-one tensor
$$
\tensorX=\ba^{(1)} \circ \ba^{(2)} \circ \dots \circ \ba^{(N)}=\llbracket 1;\ba^{(1)}, \ba^{(2)}, \dots , \ba^{(N)} \rrbracket
$$
where $1$ is the scalar one; mode-$n$ matrix product
$$
\tensorX \times_n \bA = \llbracket \tensorX;\bI_{I_1}, \dots ,\bI_{I_{n-1}}, \bA ,\bI_{I_{n+1}},\cdots, \bI_{I_{N}} \rrbracket,
$$
where $\bI_k$ is the unit matrix of dimension $k\times k$; the inner product of a tensor $\tensorG$ and a symmetric rank-one tensor (which defines a polynomial):
$$
\langle \tensorG, \underbrace{\bx \circ \cdots \circ \bx}_{N} \rangle =\llbracket \tensorG; \underbrace{\bx^\mathrm{T}, \ldots, \bx^\mathrm{T}}_{N}\rrbracket.
$$
The CP decomposition can also be viewed as a special Tucker decomposition with
$$
\tensorX = \sum_{t=1} ^r \ba^{(1,t)} \circ \cdots \circ \ba^{(N,t)} = \llbracket \tensorI; \bA^{(1)}, \ldots, \bA^{(N)}\rrbracket,
$$
where $\bA^{(n)}=[\ba^{(n,1)}, \ldots, \ba^{(n,r)}]$ for each $n=1, \ldots, N$, and $\tensorI$ is an $N$-way unit tensor, with all zero elements except
the diagonal elements are ones.

\section{The invariance of the tensor ranks}\label{sec:rank} 

There are certain correspondence between the CP decomposition of a given tensor and that of its {\it orthonormal} Tucker core \cite{Khoromskaia2010,BroAndersson1998,Khoromskij2006,TDaA}.
In particular, performing Tucker decomposition and further CP decomposition on the Tucker core is called two-level rank decomposition in \cite{Khoromskij2006}.
Approximating the original tensor by the two-level decomposition is discussed in \cite{KhoromskijKhoromskaia2007,Khoromskaia2010}.
Possible computational savings gained by exploiting this relation were discussed in  \cite{TomasiBro2006}. Moreover, 
as a direct consequence of this correspondence, the CP rank of a tensor equals to that of its orthonormal Tucker core, which is also known as the CANDELINC Theorem; see~\cite{CARROLL80}.\footnote{We would like to thank Nikos Sidiropoulos for his insightful comments and information on the topic in private communications.}

In this section we aim to show that in fact several notions on the rank of a tensor carryover to that of its {\it independent} Tucker core.
Besides, we provide a unified treatment on the independent Tucker decomposition in such a way that
the technical results presented in this section will facilitate our analysis in later discussions.

\subsection{Invariance of the CP rank under independent Tucker decomposition}
Let us start with the CP rank.
\begin{theorem}   \label{Thm:CPRankInvariance}
For any given tensor $\tensorX\in\complN N$ with independent Tucker decomposition
$
\tensorX =\llbracket \tensorG; \bX^{(1)}, \ldots, \bX^{(N)} \rrbracket,
$
we have $\rank_{CP}(\tensorX)=\rank_{CP}(\tensorG)$.
\end{theorem}


Before proving Theorem \ref{Thm:CPRankInvariance}, we shall first show the following lemma. 

\begin{lemma}\label{lemma:IndependentTuckerDecompositionSingle}
For an $N$-way tensor $\tensorX \in \complN N$ with independent Tucker decomposition
\begin{equation}\label{decomposition1}
\tensorX = \tensorG \times_1 \bA^{(1)} \times_2 \cdots \times_N \bA^{(N)},
\end{equation}
there exist $\bB^{(1)}, \bB^{(2)}, \dots, \bB^{(N)}$ such that
$$
\tensorG=\tensorX \times_1 \bB^{(1)} \times_2 \cdots \times_N \bB^{(N)}\;\mbox{and}\; \tensorG \in \complR N.
$$
Also, for any $n=1, 2, \dots, N$, we have
$$
\tensorG=\tensorG\times_n \bA^{(n)}\times_n \bB^{(n)},\mbox{~} \tensorX=\tensorX\times_n \bB^{(n)}\times_n \bA^{(n)}.
$$
\end{lemma}
\noindent {\bf Proof:} Since $\bA^{(n)}$ is a tall matrix (columns are linearly independent), $\bA^{(n)\rm H}\bA^{(n)}$ is invertible. By letting $$\bB^{(n)}=(\bA^{(n)\rm H}\bA^{(n)})^{\rm -1}\bA^{(n)\rm H}$$
we have $\bB^{(n)}\bA^{(n)} = (\bA^{(n)\rm H}\bA^{(n)})^{\rm -1}\bA^{(n)\rm H}\bA^{(n)}=\bI_{I_n}$. As a result,
$$\tensorG\times_n \bA^{(n)}\times_n \bB^{(n)}=\tensorG\times_n(\bB^{(n)}\bA^{(n)})=\tensorG \times_n \bI_{R_n}= \tensorG.$$
Moreover, since
$\tensorX = \tensorG \times_1 \bA^{(1)} \times_2 \cdots \times_N \bA^{(N)}$,
applying $\times_1 \bB^{(1)} \times_2 \cdots \times_N \bB^{(N)}$ on both sides of \eqref{decomposition1} yields:
\begin{eqnarray*}
\tensorX \times_1 \bB^{(1)} \times_2 \cdots \times_N \bB^{(N)} &=& \tensorG \times_1 \bA^{(1)} \times_2 \cdots \times_N \bA^{(N)}\times_1 \bB^{(1)} \times_2 \cdots \times_N \bB^{(N)}\\
&=& \tensorG \times_1 (\bB^{(1)} \bA^{(1)}) \times_2 \cdots \times_N (\bB^{(N)} \bA^{(N)}) = \tensorG.
\end{eqnarray*}
In a similar vein, applying $\times_n  \bB^{(n)}\times_n \bA^{(n)}$ on both sides of \eqref{decomposition1} yields
$$\tensorX\times_n  \bB^{(n)}\times_n \bA^{(n)}=\tensorG \times_1 \bA^{(1)} \times_2 \cdots \times_n \bA^{(n)}\bB^{(n)}\bA^{(n)} \dots \times_N \bA^{(N)}=\tensorX.$$
\hfill $\square$\\
Note that the above lemma leads to
an exact independent Tucker decomposition of a given tensor $\tensorX$. First, for each $n$ one performs a mode-$n$ matricization on $\tensorX$
to get $\bX_{(1)}, \cdots, \bX_{(N)}$. Then, for each mode-$n$ one computes a matrix factorization such that $\bX_{(n)} = \bA^{(n)} \bC^{(n)}$ and $\bA^{(n)}$ has full column rank.
Finally,
letting $\bB^{(n)} = \bA^{(n)} \big{(} (\bA^{(n)})^H\bA^{(n)}\big{)}^{-1}$, we have
\begin{eqnarray*}
\left(\tensorX \times_n (\bA^{(n)})^H \times_n \bB^{(n)} \right)_{(n)} &=& \left(\tensorX \times_n \big{(}\bA^{(n)} \big{(}(\bA^{(n)})^H\bA^{(n)}\big{)}^{-1}  (\bA^{(n)})^H\big{)}  \right)_{(n)} \\
& = & \bA^{(n)} \big{(}(\bA^{(n)})^H\bA^{(n)}\big{)}^{-1}  (\bA^{(n)})^H \bX_{(n)}\\
& = & \bA^{(n)} \big{(}(\bA^{(n)})^H\bA^{(n)}\big{)}^{-1}  (\bA^{(n)})^H  \bA^{(n)} \bC^{(n)} \\
& = & \bA^{(n)} \bC^{(n)} = \bX_{(n)} .
\end{eqnarray*}
Moreover, due to the one-to-one correspondence between a tensor and its mode matricization we conclude that $\tensorX \times_n (\bA^{(n)})^H \times_n \bB^{(n)} = \tensorX $.
Now, 
by letting
$
\tensorG = \tensorX \times_1 (\bA^{(1)})^H \times_2 \cdots \times_N (\bA^{(N)})^H,
$
an exact independent Tucker decomposition
$\llbracket \tensorG; \bB^{(1)}, \ldots, \bB^{(N)} \rrbracket$ of $\tensorX$ follows.


\begin{lemma}   \label{Lem:TensorProdMatrProd}
For any $N$-way tensor
$$
\tensorX = \sum_{t=1}^r \ba ^{(1,t)} \circ\ba ^{(2,t)} \circ \cdots \circ \ba^{(N,t)}
$$
assuming the multiplications are compatible
we have
$$
\tensorX \times_n \bB = \sum^r _{t=1} \ba ^{(1,t)} \circ \cdots \circ \ba ^{(n-1,t)}\circ (\bB\ba ^{(n,t)})\circ \ba ^{(n+1,t)} \circ \cdots \circ \ba ^{(N,t)}.
$$
\end{lemma}

\noindent {\bf Proof:} Denote $\bA^{(n)}=[\ba^{(n,1)}, \ldots, \ba^{(n,r)}]$. 
We have $\tensorX = \llbracket \tensorI; \bA^{(1)}, \ldots, \bA^{(N)}\rrbracket$, and so
\begin{eqnarray*}
\tensorX \times_n \bB
& = &
\llbracket \tensorI; \bA^{(1)}, \ldots, \bA^{(N)}\rrbracket  \times_n \bB
\\ & = &
\llbracket \tensorI; \bA^{(1)}, \ldots, \bA^{(n-1)}, \bB\bA^{(n)}, \bA^{(n+1)},\ldots  , \bA ^{(N)} \rrbracket
\\ & = &
\sum_{t=1}^r  \ba ^{(1,t)} \circ \cdots \circ \ba^{(n-1,t)} \circ \left( \bB\ba^{(n,t)} \right) \circ \ba^{(n+1,t)} \circ \cdots \circ \ba ^{(N,t)},
\end{eqnarray*}
which completes the proof. \hfill $\square$

\noindent Now we are ready to prove Theorem \ref{Thm:CPRankInvariance}.\\
\noindent {\bf Proof of Theorem  \ref{Thm:CPRankInvariance}:} Suppose the core tensor $\tensorG$ has a CP rank $r$ associated with
the decomposition:
$$
\tensorG = \sum_{t=1}^r \bb ^{(1,t)} \circ\bb ^{(2,t)} \circ \cdots \circ \bb^{(N,t)}.
$$
Then Lemma \ref{Lem:TensorProdMatrProd} suggests that
\begin{eqnarray*}
\tensorX &=& \tensorG \times_1 \bA^{(1)} \times_2 \cdots \times_N \bA^{(N)}\\
& = & \sum_{t=1}^r \left(\bA^{(1)} \bb ^{(1,t)}\right) \circ \cdots \circ \left(\bA^{(N)} \bb^{(N,t)}\right),
\end{eqnarray*}
which is a valid rank-1 decomposition of $\tensorX$ {with $r$ rank-1 terms}, implying that $\rank_{CP}(\tensorX) \leq r = \rank_{CP}(\tensorG)$.

On the other hand, since the Tucker decomposition is independent, Lemma \ref{lemma:IndependentTuckerDecompositionSingle} holds and there exist $\bB^{(1)}, \bB^{(2)}, \dots, \bB^{(N)}$ such that
$$
\tensorG=\tensorX \times_1 \bB^{(1)} \times_2 \cdots \times_N \bB^{(N)}.
$$
Applying the same argument, we have $\rank_{CP}(\tensorG) \leq \rank_{CP}(\tensorX)$. 
This completes the proof for Theorem \ref{Thm:CPRankInvariance}.
\hfill $\square$

Since an orthonormal Tucker decomposition is independent, as a direct consequence of above theorem, we conclude that if $\tensorG$ is the core tensor of $\tensorX$ under the orthonormal Tucker decomposition then $\rank_{CP}(\tensorX) = \rank_{CP}(\tensorG)$.

\subsection{Tensors with a symmetric structure}
In this subsection, we establish similar results for the {\df symmetric} tensors.
Formally speaking, a tensor is {\df symmetric} if the length along all the directions are equal, and the elements are invariant under any permutation of the indices, i.e.
$$
x_{i_{\sigma(1)} i_{\sigma(2)} \cdots i_{\sigma(N)}}=x_{i_1i_2\cdots i_N}
$$
where $\sigma(\cdot)$ is any given permutation function of $\{ 1, 2, \ldots, N\}$.
Symmetric tensors are well studied; see e.g.~\cite{STaSTR}. Parallel to the definitions in the preceding sections, we have:
\begin{definition}   \label{Def:SymmetricCP}
For an $N$-way symmetric tensor $\tensorX \in \compl{I\times I \times \cdots \times I}$, a {\df symmetric CP decomposition} of size $r$ is to represent $\tensorX$ as follows
\[
\tensorX = \sum_{t=1} ^r \ba^{(t)} \circ \cdots \circ \ba^{(t)}.
\]
The {\df symmetric rank} of \tensorX, denoted by $\rank_{S}(\tensorX)$, is the minimal integer $r$ such that a size-$r$ symmetric CP decomposition exists.
\end{definition}

\begin{definition}   \label{Def:SymmetricTucker}
For an $N$-way symmetric tensor $\tensorX \in \compl{I\times I \times \cdots \times I}$, an exact Tucker decomposition is called {\df symmetric} if all the factor matrices are identical; i.e., the Tucker decomposition is of the form
$$
\tensorX = \llbracket \tensorG^s;\bX, \bX, \ldots, \bX \rrbracket.
$$
As before, a symmetric Tucker decomposition is said to be {\df independent} if the factor matrix has full column rank; a symmetric Tucker decomposition is said to be {\df orthonormal} if the factor matrix has orthonormal columns.
\end{definition}
We remark that for an $N$-way symmetric tensor $\tensorX \in \compl{I \times I \times \cdots \times I}$, its matricizations of all modes are identical to each other, i.e., $\tensorX_{(m)} = \bX_{(n)}$ for all $m, n =1, \ldots, N$. Based on this observation, an independent symmetric Tucker decomposition of a given symmetric tensor $\tensorX$ can be constructed as follows.
First, perform a symmetric matrix factorization such that $\bX = \bA \bA^{H}$ and $\bA$ has full column rank.
Then, construct
$$
\tensorG = \tensorX \times_1 \bA^H \times_2 \cdots \times_N \bA^H,
$$
and $\bB = \bA \left(\bA^H \bA \right)^{-1}$. Following a similar argument for asymmetric tensors, it can be verified that
$\llbracket \tensorG; \bB, \ldots, \bB \rrbracket$
is an independent symmetric Tucker decomposition of $\tensorX$. Now we present the invariance of symmetric CP rank.

\begin{theorem}   \label{thm:SymmetricTuckerConst}
For any symmetric tensor $\tensorX \in \compl{I\times I\times \cdots \times I}$ and its independent symmetric decomposition
$$
\tensorX = \llbracket \tensorG^s;\bX, \bX, \ldots, \bX \rrbracket,
$$
the core tensor $\tensorG ^s \in \compl{J\times J\times \cdots \times J}$ is also symmetric, and $\rank_S (\tensorG ^s)=\rank_S(\tensorX)$.
\end{theorem}

\noindent {\bf Proof:}
By Lemma~\ref{lemma:IndependentTuckerDecompositionSingle} we have
\begin{equation}\label{symmetic-tucker-rep}
\tensorG^s=\tensorX \times_1 \bB \times_2 \cdots \times_N \bB,
\end{equation}
where $\bB = (\bX^H\bX)^{-1}\bX^H \in \compl{J \times I}$. Now we show that $\tensorG^s$
is symmetric. For any $(j_1 j_2 \dots j_N)$ and any permutation function $\sigma(\cdot)$ of $\{1, 2, \ldots, N\}$, by definition of the mode product, we have
\begin{eqnarray*}
g_{j_1 j_2 \dots j_N}^s
&=&
\sum_{i_1, \ldots, i_N=1} ^I x_{i_1 i_2 \dots i_N} \cdot b_{i_1j_1} b_{i_2j_2} \cdots b_{i_Nj_N}
\\  &=&
\sum_{i_1, \ldots, i_N=1} ^I x_{i_1 i_2 \dots i_N} \cdot b_{i_{\sigma(1)}j_{{\sigma(1)}}} b_{i_{\sigma(2)}j_{\sigma(2)}} \cdots b_{i_{\sigma(N)}j_{\sigma(N)}}
\\  &=&
\sum_{i_1, \ldots, i_N=1} ^I x_{i_{\sigma(1)} i_{\sigma(2)} \dots i_{\sigma(N)}} \cdot b_{i_{\sigma(1)}j_{{\sigma(1)}}} b_{i_{\sigma(2)}j_{\sigma(2)}} \cdots b_{i_{\sigma(N)}j_{\sigma(N)}} \\
&=& g_{j_{\sigma(1)}j_{\sigma(2)}\cdots j_{\sigma(N)}}^s,
\end{eqnarray*}
where the third equality follows from the fact that $\tensorX$ is symmetric.

Now that $\tensorG^s$ is symmetric, we may assume its symmetric CP rank to be $r$, with
the decomposition
$$
\tensorG^s = \sum_{t=1}^r \underbrace{ \bb ^{t} \circ \cdots \circ \bb^{t}}_{N}.
$$
Invoking Lemma \ref{Lem:TensorProdMatrProd} yields
\begin{eqnarray*}
\tensorX &=& \tensorG^s \times_1 \bX \times_2 \cdots \times_N \bX\\
& = & \sum_{t=1}^r \underbrace{\left(\bX \bb ^{t}\right) \circ \cdots \circ \left(\bX \bb ^{t}\right)}_{N},
\end{eqnarray*}
which is a symmetric rank-1 decomposition of $\tensorX$. This implies that $\rank_{S}(\tensorX) \leq \rank_{S}(\tensorG^s)$.
Noting that we can also decompose $\tensorG^s$ in the form of \eqref{symmetic-tucker-rep},
by the same argument we have $\rank_{S}(\tensorG^s) \leq \rank_{S}(\tensorX)$, leading to $\rank_{S}(\tensorX) = \rank_{S}(\tensorG^s)$.
\hfill $\square$

\section{Invariance of the tensor norms and the border rank}\label{sec:norm}

In this section, we study the invariance properties of various tensor norms and 
the so-called border rank.
These relationships enable us to measure the error between the tensor and an approximative decomposition.

\subsection{Invariance of tensor norms under independent Tucker decomposition}
We first consider the tensor Frobenius norm. 

\begin{theorem}   \label{Thm:FrobeniousNormInvariance}
For any given tensor $\tensorX\in\complN N$ with independent Tucker decomposition
$
\tensorX = 
\llbracket \tensorG; \bX^{(1)}, \ldots, \bX^{(N)} \rrbracket
$,
we have that
\begin{equation}\label{equ:FrobeniousNormInvariance}
\alpha \|\tensorG\|_F \leq \|\tensorX\|_F \leq \beta \|\tensorG\|_F,
\end{equation}
where $\beta = \|\bX^{(1)}\|_2\|\bX^{(2)}\|_2\cdots \|\bX^{(N)}\|_2$ and $\alpha = \beta/(\prod_{n=1}^N{\kappa(\bX^{(n)})})$
with $\kappa(\bX^{(n)})$ being the condition number of the matrix $\bX^{(n)}$ for all $n$.
\end{theorem}

The following lemma presents a bound on the Frobenius norm of mode-$n$ matrix multiplication, which leads to the proof of Theorem \ref{Thm:FrobeniousNormInvariance}.
\begin{lemma}   \label{Lem:FrobeniusMatrixProduct}
For any given tensor $\tensorX \in \complN N$ and matrix $\bA \in \compl{J_n \times I_n}$, we have
$$
\|\tensorX \times_n \bA \|_F \leq \|\bA\|_2 \|\tensorX  \|_F.
$$
\end{lemma}
\noindent {\bf Proof:}  The claimed inequality is equivalent to
$$
\| \bA \bX_{(n)} \|_F ^2 \leq \| \bA\|_2 ^2 \|\bX_{(n)}  \|_F ^2.
$$
To show this, let $\bX_{(n,k)}$ be the $k$-th column of $\bX_{(n)}$. We have
\begin{eqnarray*}
\| \bA \bX_{(n)} \|_F ^2
&=&
\sum_k \|\bA \bX_{(n,k)}\|^2 _F
\\ &\leq &
\sum_k \|\bA\|_2^2 \| \bX_{(n,k)}\|^2 _F
\\ &=&
\|\bA\|_2^2 \left( \sum_k \| \bX_{(n,k)}\|^2 _F \right)
\\ &=&
\|\bA\|_2^2  \| \bX_{(n)}\|^2 _F ,
\end{eqnarray*}
where the inequality is due to the consistency of the matrix spectral norm and the Euclidean vector norm. 
\hfill $\square$


\noindent {\bf Proof
 of Theorem  \ref{Thm:FrobeniousNormInvariance}:}
 Since $\tensorX = \llbracket \tensorG; \bX^{(1)}, \ldots, \bX^{(N)} \rrbracket$, by Lemma~\ref{Lem:FrobeniusMatrixProduct},
$$
 \|\tensorX \|_F=\|\llbracket \tensorG; \bX^{(1)}, \ldots, \bX^{(N)} \rrbracket \|_F \leq  \|\tensorG\|_F\|\bX^{(1)}\|_2\|\bX^{(2)}\|_2\cdots \|\bX^{(N)}\|_2.
 $$
Thus, we may simply let $\beta=\|\bX^{(1)}\|_2\|\bX^{(2)}\|_2\dots \|\bX^{(N)}\|_2$ and obtain the upper bound in~\eqref{equ:FrobeniousNormInvariance}.
On the other hand, by Lemma~\ref{lemma:IndependentTuckerDecompositionSingle} it follows that
$$
\tensorG=\llbracket \tensorX; \bY^{(1)}, \ldots, \bY^{(N)} \rrbracket ,
$$
with $\bY^{(n)} = \left((\bX^{(n)})^H\bX^{(n)}\right)^{-1}(\bX^{(n)})^H$. Therefore, from Lemma~\ref{Lem:FrobeniusMatrixProduct} one has
$$
 \|\tensorG \|_F=\|\llbracket \tensorX; \bY^{(1)}, \ldots, \bY^{(N)} \rrbracket \|_F \leq  \|\tensorX\|_F\|\bY^{(1)}\|_2\|\bY^{(2)}\|_2 \cdots \|\bY^{(N)}\|_2.
$$
Now we let $\alpha=\frac{1}{\|\bY^{(1)}\|_2\|\bY^{(2)}\|_2\cdots \|\bY^{(N)}\|_2}$, and it remains to prove $\beta/\alpha = \prod_{n=1}^N{\kappa(\bX^{(n)})}$.
To this end, for any $n$, suppose the matrix $\bX^{(n)}$ has SVD: $\bX^{(n)} = \bU^{(n)} \bSigma (\bV^{(n)})^H$. Then
\begin{eqnarray*}
\bY^{(n)} &=& \left((\bX^{(n)})^H\bX^{(n)}\right)^{-1}(\bX^{(n)})^H\\
& = & \left(\bV^{(n)} \bSigma^2 (\bV^{(n)})^H\right)^{-1}\bV^{(n)} \bSigma (\bU^{(n)})^H\\
& = & \bV^{(n)} \bSigma^{-2} (\bV^{(n)})^H \bV^{(n)} \bSigma (\bU^{(n)})^H\\
& = & \bV^{(n)} \bSigma^{-1} (\bU^{(n)})^H.
\end{eqnarray*}
{Therefore, $\|\bY^{(n)}\|_2 =\frac{1}{ \sigma_{\min} (\bX^{(n)})} $ for all $n$ and $\alpha= \prod_{n=1}^N{\sigma_{\min} (\bX^{(n)})}$.} Consequently,
$$\frac{\beta}{\alpha} = \frac{\prod_{n=1}^N{\|(\bX^{(n)})\|_2}}{\prod_{n=1}^N{\sigma_{\min} (\bX^{(n)})}}= \frac{\prod_{n=1}^N{\sigma_{\max} (\bX^{(n)})}}{\prod_{n=1}^N{\sigma_{\min} (\bX^{(n)})}} = \prod_{n=1}^N{\kappa(\bX^{(n)})}.$$
\hfill $\square$

Note that when $\beta = \alpha$, the estimation (\ref{equ:FrobeniousNormInvariance}) on the Frobenius norm of $\tensorX$ becomes exact.
In general, the ratio $\beta/\alpha$ measures the {quality of approximating the Frobenius norm of ${\cal X}$ by that of $\alpha\, {\cal G}$}. When the factor matrix is orthonormal, the associated condition number is $1$, and then the Frobenius norm of the tensor and its Tucker core are equal.
\begin{corollary}
For an orthonormal Tucker decomposition, its factor
matrix $\bX^{(n)}$ is orthonormal. Thus $\|\bX^{(n)}\|_2=\|(\bX^{(n)})^{\mathrm H}\|_2=1$ for $n =1,2,..., N$, and we have $\|\tensorX\|_F=\|\tensorG\|_F$.
\end{corollary}

Remark that the above result was shown by De Lathauwer et al.\ as Property~8 in \cite{DeLathauwerMoorVandewalle2000}.


\subsection{The quasi-$p$ norm and the tensor nuclear norm}

We proceed to 
other tensor norms in this subsection.
\begin{definition}
For any tensor $\tensorX\in \complN N$, the {\df tensor $p$-quasi norm} for $0< p \le 1$ \footnote{In a private conversation, Lek-Heng Lim pointed out to us that $\|\tensorX \|_p$ trivially equals to zero for any tensor $\tensorX$ for any $p > 1$.}
of $\tensorX$ is defined as
\begin{equation}
\begin{split}
\|\tensorX \|_p\triangleq {\inf}\bigg{\{}&
\left(\sum_{s=1}^r |\lambda_s|^p\right)^{1/p}
:\: \tensorX=\sum_{s=1}^r{\lambda_s \bx_s^{(1)}\circ \bx_s^{(2)}\circ} \dots \bx_s^{(N)},\\
&\|\bx_s^{(n)}\| =1,\;\forall s=1, 2, \dots, r,\; \forall n=1, 2, \dots , N \bigg{\}} .
\end{split}
\end{equation}
\end{definition}
When $N=2$, it reduces to the Schatten $p$-quasi norm for matrix, which plays an important role in low-rank matrix optimization~\cite{JiSzeZhouSoYe2013}.
When $p = 1$, the above definition corresponds to the tensor nuclear norm, which was originally proposed in
Grothendieck~\cite{Grothendieck1955} and Schatten~\cite{Schatten1950}. Recently, this tensor nuclear norm was applied in Yuan and Zhang~\cite{YuanZhang2015} to analyze the statistical properties of
Tensor completion problem. Friedland and Lim~\cite{FriedlandLim2014} showed that computing the nuclear norm of a given tensor is NP-hard.
Interestingly, for the quasi-$p$ norm, similar bounds between a given tensor and its core tensor hold true.

\begin{theorem}
For any given tensor $\tensorX\in\complN N$ with independent Tucker decomposition
$
\tensorX = 
\llbracket \tensorG; \bX^{(1)}, \ldots, \bX^{(N)} \rrbracket
$,
we have
$$\alpha \|\tensorG\|_p \leq \|\tensorX\|_p \leq \beta \|\tensorG\|_p,$$
where $\beta = \|\bX^{(1)}\|_2\|\bX^{(2)}\|_2\dots \|\bX^{(N)}\|_2$ and $\alpha = \beta/(\prod_{n=1}^N{\kappa(\bX^{(n)})})$
with $\kappa(\bX^{(n)})$ being the condition number of matrix $\bX^{(n)}$ for all $n$.
\end{theorem}

\noindent {\bf Proof:}
{For any $\epsilon>0$ find CP decomposition of tensor $\tensorG=\sum_{s=1}^r{\lambda_s \bg_s^{(1)}\circ \bg_s^{(2)}\circ \dots \bg_s^{(N)}}$ with $\|\bg_s^{(n)}\| = 1 $ satisfying  $\|\tensorG\|_p \ge  \left(\sum_{s=1}^r |\lambda_s|^p\right)^{1/p} - \epsilon$, denote the index set $S = \{ 1 \le s \le r \; | \; \bX^{(i)}\bg_s^{(i)}=0 \;\mbox{for some}\;i,\; 1\le i \le N \}$. Then from Lemma~\ref{Lem:TensorProdMatrProd}, we have
\begin{eqnarray*}
\tensorX&=& 
\sum_{s=1}^r{\lambda_s \bg_s^{(1)}\circ \bg_s^{(2)}\circ \dots \bg_s^{(N)}}\times_1 \bX^{(1)}\times_2 \bX^{(2)} \times_3 \cdots  \times_N \bX^{(N)} \\
&=&\sum_{s=1}^r{\lambda_s (\bX^{(1)}\bg_s^{(1)})\circ (\bX^{(2)}\bg_s^{(2)})\circ \dots \circ(\bX^{(N)}\bg_s^{(N)})}\\
&=&\sum_{s \in S}{\lambda_s (\bX^{(1)}\bg_s^{(1)})\circ (\bX^{(2)}\bg_s^{(2)})\circ \dots \circ(\bX^{(N)}\bg_s^{(N)})}\\
&=& \sum_{s \in S}{\prod_{n=1}^N \|\bX^{(n)}\bg_s^{(n)}\|}\cdot \lambda_s\left(\frac{\bX^{(1)}\bg_s^{(1)}}{\|\bX^{(1)}\bg_s^{(1)}\|}\right)\circ  \dots \circ \left(\frac{\bX^{(N)}\bg_s^{(N)}}{\|\bX^{(N)}\bg_s^{(N)}\|}\right),
\end{eqnarray*}
which is a valid CP decomposition of $\tensorX$.
Moreover, by the definition of tensor $p$-quasi norm,
\begin{eqnarray*}
\|\tensorX \|_p &\le&
\left(\sum_{s \in S} {\prod_{n=1}^N \|\bX^{(n)}\bg_s^{(n)}\|^p}\cdot  |\lambda_s|^p\right)^{1/p}
\le  
\left(\sum_{s\in S} \prod_{n=1}^N (\|\bX^{(n)}\|_2^p\|\bg_s^{(n)}\|^p)\cdot|\lambda_s|^p
\right)^{1/p} \\
&=&
\left(\sum_{s=1}^r \prod_{n=1}^N (\|\bX^{(n)}\|_2^p\|\bg_s^{(n)}\|^p)\cdot|\lambda_s|^p
\right)^{1/p}
 =  \prod_{n=1}^N \|\bX^{(n)}\|_2\left(\sum_{s=1}^r |\lambda_s|^p\right)^{1/p} = \prod_{n=1}^N \|\bX^{(n)}\|_2 \cdot (\|\tensorG \|_p + \epsilon),
\end{eqnarray*}
where the second equality is due to $\|\bg_s^{(n)}\| = 1$ for all $n$ and $s$.
Since $\epsilon>0$ can by chosen arbitrarily, we have $ \|{\cal X}\|_p \le \prod_{n=1}^N \|\bX^{(n)}\|_2 \cdot \|\tensorG \|_p $.
On the other hand, due to the nature of independent Tucker decomposition and Lemma~\ref{lemma:IndependentTuckerDecompositionSingle}, one has that
$\tensorG=\llbracket \tensorX; \bY^{(1)}, \ldots, \bY^{(N)} \rrbracket$. By repeating above argument, we have $\|\tensorG \|_p \le \prod_{n=1}^N \|\bY^{(n)}\|_2 \cdot \|\tensorX \|_p$.
The rest of the proof follows similarly from that of Theorem~\ref{Thm:FrobeniousNormInvariance}.
}
\hfill $\square$

Similar to the analysis in the previous subsection, when the Tucker decomposition is orthonormal then the $p$-quasi norm of a tensor and that of its core are equal.
\begin{corollary}
For an orthonormal Tucker decomposition, we have $\|\tensorX\|_p=\|\tensorG\|_p$.
\end{corollary}

\subsection{The border rank}
Although the CP rank is a natural extension of matrix rank, one undesirable theoretical property of this definition is that the best rank-$r$
approximation may not even exist (see~\cite{TDaA} for more details). In particular, a rank-$r$ tensor may be approximated arbitrarily close by a sequence of tensors whose CP ranks are strictly less than $r$.
To get around this,  
Bini~\cite{Bini2007} proposes the concept of the border rank,
which is defined as follows.

\begin{definition}   \label{Def:BorderRank}
For an $N$-way tensor $\tensorX \in \complN N$, its {\df border rank}, denoted by $\rank_B(\tensorX)$, is defined as
\begin{equation*}   \label{Eq:DefBorderRank}
\rank_B = \min \left\{ r\;|~ \forall \epsilon >0, \exists\; \tensorE \in \complN N, \text{ s.t. } \|\tensorE \|_F \leq \epsilon \text{ and } \rank_{CP}(\tensorX +\tensorE) \leq r \right\}.
\end{equation*}
\end{definition}
In other words, the border rank of a given tensor is
the minimum {CP-rank of tensors that can be found in any neighborhood of the given tensor}.
By the bound on the tensor norms, we now establish the equality of the border rank between a given tensor and its core.

\begin{theorem}   \label{Thm:BRankInvariance}
For any given tensor $\tensorX\in\complN N$ with independent Tucker decomposition
$
\tensorX = 
\llbracket \tensorG; \bX^{(1)}, \ldots, \bX^{(N)} \rrbracket,
$
we have $\rank_{B}(\tensorX)=\rank_{B}(\tensorG)$.
\end{theorem}

\noindent {\bf Proof: }
Assume $\rank_{B}(\tensorG) = r$, we want to show $\rank_{B}(\tensorX)\le \rank_{B}(\tensorG)=r$. By definition of the border rank, for any $\epsilon >0$, there exists
{ $\tensorE \in  \mathbb{C}^{J_1\times J_2\times \cdots\times J_N}$} such that $\| \tensorE \|_F \leq \epsilon / \prod_{n=1}^N{\sigma_{\max} (\bX^{(n)})}$
and $\rank_{CP}(\tensorG +\tensorE) \leq r $.
Now, construct
$$
\tensorT =\tensorE \times_1 \bX^{(1)} \times_2 \cdots \times_N  \bX^{(N)},
$$
and consider the tensor $\tensorX + \tensorT = (\tensorG + \tensorE) \times_1 \bX^{(1)} \times_2 \cdots \times_N  \bX^{(N)}$. Obviously, by Theorem \ref{Thm:CPRankInvariance} we have
\begin{equation}\label{border-rank-formula1}
{\rank_{CP}(\tensorG +\tensorE) = \rank_{CP}(\tensorX +\tensorT) \leq r.}
\end{equation}
Moreover, according to Theorem \ref{Thm:FrobeniousNormInvariance} we also have
\begin{equation}\label{border-rank-formula2}\| \tensorT\|_F=\prod_{n=1}^N{\sigma_{\max} (\bX^{(n)})} \cdot \| \tensorE\|_F \leq \epsilon.
\end{equation}
Combining~\eqref{border-rank-formula1} and \eqref{border-rank-formula2} implies that $\rank_{B}(\tensorX)\le \rank_{B}(\tensorG)$.

Now we recall that for independent Tucker decomposition,
$
\tensorG = 
\llbracket \tensorX; \bY^{(1)}, \ldots, \bY^{(N)} \rrbracket,
$
with $\bY^{(n)} = \big{(}(\bX^{(n)})^H\bX^{(n)}\big{)}^{-1}(\bX^{(n)})^H $ for all $n$ (see Lemma~\ref{lemma:IndependentTuckerDecompositionSingle} for more details).
The inequality
$\rank_{B}(\tensorX)\ge \rank_{B}(\tensorG)$ follows similarly from the argument above, which
establishes the desired equality.
\hfill $\square$

\section{Invariance of tensor eigenvalues}\label{sec:eigenvalue}
In this section, we focus 
on the real-field and investigate the invariance properties of various notions of tensor eigenvalues.

\subsection{Invariance of the Z-eigenvalues}
Let us consider the class of real-valued symmetric tensors, and denote the symmetric rank-one tensor $\underbrace{\bx \circ \cdots \circ \bx}_{N}$ as $\bx^{\circ N}$.
For any symmetric tensor $\tensorT \in  \real{I \times \cdots \times I}$ and $(N-1)$-way rank-$1$ tensor $\bx^{\circ (N-1)}$, $\tensorT(\bx^{\circ (N-1)})$ denotes an $I$ dimensional vector such that
\begin{equation*}
(\tensorT(\bx^{\circ (N-1)}))_{i_N} \triangleq \sum_{i_1, i_2, \dots, i_{N-1}=1}^I {T_{i_1i_2\cdots i_{N-1}i_N}\cdot x_{i_1} x_{i_2} \cdots  x_{i_{N-1}}}.
\end{equation*}
With this notion in place, the Z-eigenvalue and Z-eigenvector of a tensor are defined as follows.
\begin{definition}\label{Def:Eigen}
For an $N$-way symmetric tensor $\tensorT \in \real{I \times \cdots \times I}$, if there exists a number $\lambda \in \real{}$ and a nonzero vector $\bx \in \real I$ such that
\begin{equation}
\tensorT(\bx^{\circ (N-1)})=\lambda \bx, ~\bx^{\mathrm T}\bx=1.
\end{equation}
Then $\lambda$ is called the {\df Z-eigenvalue} of $\tensorT$, and $\bx$ is called the corresponding {\df Z-eigenvector}.
\end{definition}
The Z-eigenvalues were first studied by Qi~\cite{Qi2005} and Lim~\cite{SVaEoTAVA} independently.
The relationship between the Z-eigenvalues of a symmetric tensor and that of its core is described as follows.

\begin{theorem}   \label{thm:TuckerZEigenConst}
For any given $N$-way symmetric tensor $\tensorT \in \real{I\times I\times \cdots \times I}$ with exact independent symmetric Tucker decomposition
$\tensorT = \llbracket \tensorG^s; \bX, \bX, \ldots, \bX\rrbracket$, construct
$$
\hat \tensorG^s = \tensorG^s \times_1 (\bX^{\mathrm T}\bX)^{1/2}\times_2 \cdots \times_N (\bX^{\mathrm T}\bX)^{1/2}.
$$
Then any Z-eigenvalues of $\hat \tensorG^s$ are also Z-eigenvalues of $\tensorT$ while any non-zero Z-eigenvalue of $\tensorT$ are also Z-eigenvalues of $\hat \tensorG^s$.
\end{theorem}

\noindent {\bf Proof:}
Since the Tucker decomposition is independent, by Lemma~\ref{lemma:IndependentTuckerDecompositionSingle},
$$
\tensorG^s = \tensorT \times_1 \bY \times_2 \cdots \times_{N} \bY,
$$
with $\bY = (\bX^{\mathrm T}\bX)^{-1}\bX^{\mathrm T}$.
Furthermore, due to Theorem~\ref{thm:SymmetricTuckerConst}, $\tensorG^s$ is symmetric and so is $\hat \tensorG^s$.
Let $\lambda$ be a Z-eigenvalue of $\hat \tensorG^s$ with Z-eigenvector $\ba$. We have
\begin{eqnarray*}
\lambda \ba &=& \hat \tensorG^s(\ba^{\circ (N-1)}) =
\hat \tensorG^s \times_1 \ba^\mathrm{T} \times_2 \cdots \times_{N-1} \ba^\mathrm{T}
\\&=&
\tensorG^s \times_1 (\bX^{\mathrm T}\bX)^{1/2}\times_2 \cdots \times_N (\bX^{\mathrm T}\bX)^{1/2}  \times_1 \ba^\mathrm{T} \times_2 \cdots \times_{N-1} \ba^\mathrm{T}\\
&=&
\tensorT \times_1 \bY \times_2 \cdots \times_{N} \bY \times_1 \ba^\mathrm{T}(\bX^{\mathrm T}\bX)^{1/2}\times_2 \cdots \times_{N-1} \ba^\mathrm{T}(\bX^{\mathrm T}\bX)^{1/2} \times_{N} (\bX^{\mathrm T}\bX)^{1/2}
\\&=&
\tensorT \times_1 \ba^\mathrm{T}(\bX^{\mathrm T}\bX)^{-1/2}\bX^\mathrm{T} \times_2 \cdots \times_{N-1} \ba^\mathrm{T}(\bX^{\mathrm T}\bX)^{-1/2}\bX^\mathrm{T}  \times_{N} (\bX^{\mathrm T}\bX)^{-1/2}\bX^\mathrm{T}.
\end{eqnarray*}
Applying $\times_N \bX (\bX^{\mathrm T}\bX)^{-1/2}$ to both sides and invoking Lemma \ref{lemma:IndependentTuckerDecompositionSingle} yield
\begin{eqnarray*}
\tensorT((\bX (\bX^{\mathrm T}\bX)^{-1/2}\ba)^{\circ (N-1)})&=&\tensorT \times_1\ba^\mathrm{T}(\bX^{\mathrm T}\bX)^{-1/2}\bX^{\mathrm T} \times_2 \cdots \times_{N-1}\ba^\mathrm{T}(\bX^{\mathrm T}\bX)^{-1/2}\bX^{\mathrm T}\\
&=&\lambda \ba \times_N \bX (\bX^{\mathrm T}\bX)^{-1/2} = \lambda  \bX (\bX^{\mathrm T}\bX)^{-1/2} \ba.
\end{eqnarray*}
Furthermore, we have $\ba^{\rm T}(\bX^{\mathrm T}\bX)^{-1/2}\bX^{\mathrm T}\bX (\bX^{\mathrm T}\bX)^{-1/2} \ba=\ba^{\rm T}\ba=1$, thus $\lambda$ is an eigenvalue of $\tensorT$ with corresponding eigenvector $\bX (\bX^{\mathrm T}\bX)^{-1/2}  \ba$.

On the other hand, suppose $\mu$ is a Z-eigenvalue of $\tensorT$ associated with the Z-eigenvector $\bb$. Following a similar argument, one has
\begin{eqnarray*}\label{equ:Z-eigenvalueProof}
&&\hat \tensorG^s\left(\big{(}(\bX^{\mathrm T}\bX)^{-1/2} \bX^{\rm T}\bb\big{)}^{\circ (N-1)}\right)\\
&=&\hat \tensorG^s \times_1(\bb^{\rm T}\bX (\bX^{\mathrm T}\bX)^{-1/2})\times_2 \cdots \times_{N-1}(\bb^{\rm T}\bX (\bX^{\mathrm T}\bX)^{-1/2})\\
&=&\mu (\bX^{\mathrm T}\bX)^{-1/2} \bX^{\rm T} \bb.
\end{eqnarray*}
By applying $\times_N \bb^{\rm T}\bX (\bX^{\mathrm T}\bX)^{-1/2}$ to both sides of the above equality, we have
\begin{equation*}
\hat \tensorG^s \times_1(\bb^{\rm T}\bX (\bX^{\mathrm T}\bX)^{-1/2})\times_2 \cdots \times_{N}(\bb^{\rm T}\bX (\bX^{\mathrm T}\bX)^{-1/2})
= \mu \bb^{\rm T}\bX (\bX^{\mathrm T}\bX)^{-1} \bX^{\rm T} \bb.
\end{equation*}
Moreover, it is easy to verify that
\begin{eqnarray*}&&\hat \tensorG^s \times_1(\bb^{\rm T}\bX (\bX^{\mathrm T}\bX)^{-1/2})\times_2 \cdots \times_{N}(\bb^{\rm T}\bX (\bX^{\mathrm T}\bX)^{-1/2})\\
&=& \tensorG^s \times_1 \bb^{\rm T}\bX  \times_2 \cdots \times_N \bb^{\rm T}\bX \\
&=&\tensorT \times_1 \bY \times_2 \cdots \times_{N} \bY \times_1\bb^{\rm T}\bX\times_2\dots \times_N\bb^{\rm T}\bX\\
&=&\tensorT \times_1 \bb^{\rm T} \times_2 \cdots \times_{N} \bb^{\rm T} = \mu \bb\times_N\bb^{\rm T}=\mu,
\end{eqnarray*}
where the third equality is due to Lemma \ref{lemma:IndependentTuckerDecompositionSingle}.
If $\mu\neq0$, then $\bb^{\rm T}\bX (\bX^{\mathrm T}\bX)^{-1} \bX^{\rm T} \bb=1$, meaning that $\mu$ is the Z-eigenvalue of core tensor $\tensorG^s$ with the associated Z-eigenvector $(\bX^{\mathrm T}\bX)^{-1/2} \bX^{\rm T} \bb$.
\hfill $\square$


We remark that if the Tucker decomposition is orthonormal then $\hat \tensorG^s = \tensorG^s$, and the above theorem states that
all the Z-eigenvalues (except zero) of a symmetric tensor equal to the Z-eigenvalues of its core. This motivates us to focus on the eigenvalues of the core tensor, which however may miss a zero eigenvalue. Fortunately, the following result tells us that when the size of the core is strictly less than the size of the original tensor, then zero eigenvalue is always present.

\begin{proposition}\label{prop:zero-Z-eig}
Suppose a given $N$-way symmetric tensor $\tensorT \in \real{I\times I\times \cdots \times I}$ has an exact independent symmetric Tucker decomposition
$\tensorT = \llbracket \tensorG^s; \bX, \bX, \ldots, \bX\rrbracket$ such that $\tensorG^s \in \real{J\times J\times \cdots \times J}$.
If $I>J$, then $0$ is an Z-eigenvalue.
\end{proposition}
\noindent {\bf Proof:}
We note the the factor matrix $\bX\in \real{I \times J}$ in the independent Tucker decomposition.
Since $J<I$, there exists a non-zero vector $\ba$ such that $\bX^{\rm T}\ba=0$. Thus we have
$$
\tensorT \times_1\ba^{\rm T}\dots \times_{N-1}\ba^{\rm T}=\tensorG^s\times_1 (\ba^{\rm T}\bX )\dots\times_{N-1} (\ba^{\rm T}\bX ) \times_N\bX=0,
$$
implying that $\ba$ is a Z-eigenvector corresponding to the Z-eigenvalue $0$ of the tensor $\tensorT$.
\hfill $\square$

Invariance of the Z-eigenvalues has interesting implications regarding the nonnegativity properties as well. A commonly used notion of nonnegativity
is the {\it positive semidefinite
(PSD)} tensor:
$$
\tensorT (\bx^{\circ (2N)}) \ge 0,\;\forall\; \bx \in \real{I},
$$
where $\tensorT $ is symmetric and has degree $2N$. Since all the Z-eigenvalues of $\tensorT$ correspond to all the KKT points of the polynomial optimization:
$\min_{\|\bx\|_2 =1}\tensorT (\bx^{\circ (2N)})$, we have the following result as a consequence of Theorem~\ref{thm:TuckerZEigenConst}:
\begin{corollary}
For any given $2N$-way symmetric tensor $\tensorT \in \real{I\times I\times \cdots \times I}$ with exact independent symmetric Tucker decomposition
$\tensorT = \llbracket \tensorG^s; \bX, \bX, \ldots, \bX\rrbracket$, $\tensorT$ is PSD if and only if $\tensorG^s$ is PSD.
\end{corollary}
The dual of the class of PSD tensors is the {\it sum of
powers (SOP)} tensors (see \cite{JiangLiZhang2015}):
$$\tensorT = \sum_{k=1}^{m} (\bx^k)^{\circ (2N)},\;\mbox{where}\;m\;\mbox{a positive integer and}\;\bx^k \in \real{I},\;\forall\;k=1,\cdots, m  .$$
Similarly, Theorem~\ref{Thm:CPRankInvariance} also implies:
\begin{corollary} For any given $2N$-way symmetric tensor $\tensorT \in \real{I\times I\times \cdots \times I}$ with exact independent symmetric Tucker decomposition
$\tensorT = \llbracket \tensorG^s; \bX, \bX, \ldots, \bX\rrbracket$, $\tensorT$ is SOP if and only if $\tensorG^s$ is SOP.
\end{corollary}
We remark that verifying whether a tensor is PSD or SOP are in general NP-hard problems~\cite{HillarLim2013,JiangLiZhang2015}. A famous result of Hilbert states that a 4th order tertiary polynomial is PSD if and only if it is a sum of squares, where the latter condition can be verified easily. This implies that if the core of a 4th order symmetric tensor has size no more than 3, then one can easily verify if it is PSD or not.

\subsection{Invariance of the M-eigenvalue}

In this subsection, we consider tensors with a ``less'' symmetric structure, termed {\it partial symmetricity}. In particular,
for a four-way tensor $\tensorT\in\compl{N\times M\times N\times M}$, we call it
{\it partial symmetric} if
$$
{t_{ijkl}=t_{kjil}=t_{ilkj}=t_{klij}} \mbox{ , for }i,k=1, 2, \dots, N; j,l=1, 2, \dots, M.
$$
Similarly, a Tucker decomposition is called {\df partial symmetric} if it is of the form
 $$
 \tensorT=\llbracket \tensorG^{ps}; \bA, \bB, \bA, \bB \rrbracket.
 $$
Below we introduce the notion of M-eigenvalue and M-eigenvector proposed in Qi, Dai and Han in~\cite{Qi2009}.
\begin{definition}
For a four-way partial symmetric tensor $\tensorT \in \real{N\times M\times N\times M}$, if there exist two numbers $\lambda \mbox{ and } \mu \in \real{}$, two nonzero vectors $\bx \in \real N$ and $\by \in \real M$ such that
\begin{eqnarray*}
\tensorT (\cdot, \by, \bx, \by)=\lambda \bx,~\bx^{\mathrm T} \bx=1\\
\tensorT (\bx, \by, \bx, \cdot)=\mu \by,~ \by^{\mathrm T} \by=1
\end{eqnarray*}
where {$\tensorT (\cdot, \by, \bx, \by)=\sum_{k=1}^{N}\sum_{j,l=1}^{M}t_{ijkl}y_j x_k y_l$} and {$\tensorT (\bx, \by, \bx, \cdot)=\sum_{i, k=1}^{N}\sum_{j=1}^{M}t_{ijkl}x_i y_j x_k$}. Then $\lambda$ and $\mu$ are called the {\df M-eigenvalues} of $\tensorT$, while $\bx$ and $\by$ are called the corresponding {\df M-eigenvectors}.
\end{definition}

\begin{theorem}
For any four-way partial symmetric tensor $\tensorT \in \real{N\times M\times N \times M}$ with its exact independent partial symmetric Tucker decomposition $\tensorT = \llbracket \tensorG^{ps}; \bA, \bB, \bA,  \bB\rrbracket$, construct
$$
\hat \tensorG^{ps} = \times_1 (\bA^{\mathrm T}\bA)^{1/2} \times_2 (\bB^{\mathrm T}\bB)^{1/2} \times_3 (\bA^{\mathrm T}\bA)^{1/2} \times_4 (\bB^{\mathrm T}\bB)^{1/2} .
$$
Then any M-eigenvalues of $\hat \tensorG^{ps}$ are also M-eigenvalues of $\tensorT$ while any non-zero M-eigenvalues of $\tensorT$ are also M-eigenvalues of $\hat \tensorG^{ps}$.
\end{theorem}

\noindent {\bf Proof:}
Similar to the proof of Theorem~\ref{thm:SymmetricTuckerConst}, one can show that $\tensorG^{ps} $ and $\hat \tensorG^{ps}$ are both partial symmetric.
Therefore, the M-eigenvalues of $\hat \tensorG^{ps}$ are well-defined.
The rest of the proof is similar to that of Theorem \ref{thm:TuckerZEigenConst}, and is omitted here for brevity.
\hfill $\square$

Similar to the symmetric case, if the partial symmetric Tucker decomposition is orthonormal then $\hat \tensorG^{ps} = \tensorG^{ps}$, and the equivalence of the M-eigenvalues (except for 0) between a partial symmetric tensor and its Tucker core can be established. The following proposition demonstrates when 0 is always an eigenvalue of the original tensor.

\begin{proposition}\label{prop:zero-M-eig}
Suppose a given four-way partial symmetric tensor $\tensorT\in\real{I_1\times I_2\times I_1 \times I_2}$ has an exact independent partial symmetric Tucker decomposition $\tensorT = \llbracket \tensorG^{ps}; \bA, \bB, \bA,  \bB\rrbracket$ such that $\tensorG^{ps}\in\real{J_1\times J_2\times J_1 \times J_2}$.
Either $J_1 < I_1$ or $J_2 < I_2$ implies the existence of a zero M-eigenvalue.
\end{proposition}
The proof is almost identical to that of Proposition~\ref{prop:zero-Z-eig}, and is omitted here.





\section{Miscellaneous discussions and error estimations}
\label{sec:error}

The invariance properties that we have established immediately lead to possible enhancements of many existing bounds. For instance, it is well known that for
$\tensorX \in \complN N$ with $I_1 \leq I_2 \leq \cdots \leq I_N$, it holds that $\rank_{CP}(\tensorX) \leq I_1  I_2 \cdots I_{N-1} $.
Recently, Hu~\cite{ShenglongHu} showed that
the tensor nuclear norm $\| \tensorX \|_1$ is upper bounded by $I_1  I_2 \cdots I_{N-1} \cdot \| \bX_{(N)} \|_*$, where $\| \cdot \|_*$ denotes the nuclear norm of a matrix, and the bound is tight when $N = 3$. Now, all these bounds can be sharpened by means of the Tucker rank. 

Suppose $\tensorX \in \complN N$ is an $N$-way tensor with $\rank_T(\tensorX)=(R_1, \ldots, R_N)$. Without loss of generality, assume that $R_1 \leq R_2 \leq \cdots \leq R_N$. Then
\begin{itemize}

\item
$\rank_{CP}(\tensorX) \leq R_1  R_2 \cdots R_{N-1};$

\item
$\|\tensorX\|_1 \leq R_1  R_2 \cdots R_{N-1}\cdot \|\bX_{(N)}\|_*$.


\end{itemize}


As discussed earlier, tensor related computations such as the CP decompositions, norms or eigenvalues are mostly NP hard~\cite{HillarLim2013}. Moreover, exact solution methods, such as the SOS (Sum of Squares) approach to tensor eigenvalue computation (see~\cite{AREoST}), are often very sensitive to the size of the underlying tensor. At the same time, the Tucker decomposition involves only matrix operations, hence easy computable. Therefore it is natural to consider a reduction scheme where the tensor computation is only carried out for its core. Before proceeding, we set out to explore if size of the Tucker core of a tensor is indeed typically smaller than the size of the tensor itself.
To this end, we find it compelling to test the size reduction on some well studied specific instances of tensors. Below is a summary of our experimental results. 


\begin{itemize}

\item

{\it (A tensor case studied in \cite{TGaA}.) }


This specific tensor is in $\compl{3\times 3\times 3}$, corresponding to the following polynomial
\begin{equation*}
\begin{split}
\tensorT(\bx^4)&=81x_0^4+17x_1^4+626x_2^4-144x_0x_1^2x_2+216x_0^3x_1-108x_0^3x_2+216x_0^2x_1^2+54x_0^2x_2^2+\\
&96x_0x_1^3-12x_0x_2^3-52x_1^3x_2+174x_1^2x_2^2-508x_1x_2^3+72x_0x_1x_2^2-216x_0^2x_1x_2.
\end{split}
\end{equation*}
The dimension of the tensor is $3$ while the
size of its core tensor is $2$.

\item

{\it (A tensor case studied in \cite{Nie2015}.) }


This specific tensor is in $\compl{5\times 5\times 5}$, with its components given by
$$
\tensorT_{i_1i_2i_3}=i_1i_2i_3-i_1-i_2-i_3 \ (0\leq i_1, i_2, i_3 \leq 4).
$$
The dimension of the tensor is $5$ while the size of its core tensor is $2$.

\item

{\it (Another tensor case studied in \cite{Nie2015}.) }

This specific tensor is in $\compl{5\times 5\times 5\times 5}$, with its components given by
$$
\tensorT_{i_1i_2i_3i_4}=\tan(i_1i_2i_3i_4) \ (0\leq i_1, i_2, i_3, i_4\leq 4).
$$
The dimension of the tensor is $5$ while the size of the core tensor is $4$.

\end{itemize}

For some special classes of tensors, it might be possible to estimate the size of its Tucker core.

\begin{proposition}
Consider $N$-th order tensor $\mathcal{X}$ with a separate structure:
\begin{equation}\label{seperate-structure}
\mathcal{X}_{i_1\cdots i_N} = \sum_{n=1}^{N}f_n(i_n).
\end{equation}
Then the size of its Tucker core is no more than $(2,\cdots, 2)$.
Consider the following $N$-th order symmetric real tensor $\tensorT$ (see~\cite{NieWang2014} or Example 4.12 in~\cite{AREoST}):

$$
\tensorT_{i_1\cdots,i_N} =
 {\sin(i_1 + \cdots + i_N)}.
$$

The size of its Tucker core is no more than $2$.
\end{proposition}
\begin{proof} For tensor $\mathcal{X}$, it suffices to show that the rank of any mode-$n$ matricization $\bX_{(n)}$ is no more than  $2$.
Note that $\bX_{(n)}$ can be specified componentwise by $(\bX_{(n)})_{i_n, j} = f_n(i_n)+ \sum_{k \neq n}f_k(i_k)$ where
$j=1+\sum_{k=1, k\neq n}^N{(i_k-1)J_k}$ with $\ J_k=\prod_{m=1, m\neq n}^{k-1}{I_m}$. By constructing vectors $\ba = [f_n(i_n) ]$ and $\bb =  [\sum_{k \neq n}f_k(i_k) ]$, it follows that
$\bX_{(n)} = \ba \be^{\mathrm T} + \be \bb^{\mathrm T}$. Consequently, $\rank (\bX_{(n)}) \le 2$ and the first half of the conclusion is proved.

To study tensor $\tensorT$, due to the symmetric property, without loss of generality it suffices to consider the mode-$1$ matricization $\bT_{(1)}$ such that

\begin{eqnarray*}
(\bT_{(1)})_{i_1, j} &=& {\sin(i_1+\cdots +i_N)} \\
&=& \sin(i_1) {\cos(i_2 + \cdots + i_N)} + \cos(i_1) {\sin(i_2 +\cdots + i_N)}\; \mbox{with}\; j=1+\sum_{k=2}^N(i_k-1)n^{k-1},
\end{eqnarray*}

where $n$ is the length of the tensor along each direction. Simply letting
$$\ba = [\sin(i_1) ],\; \bb =  [{\cos(i_2+\cdots+ i_N) }],\; \bc = [\cos(i_1) ]\; \mbox{and}\; \bd = [{\sin(i_2+\cdots+ i_N)} ]$$
yields that
$$
\bT_{(1)} = \ba \bb^{\mathrm T} + \bc \bd^{\mathrm T},
$$
proving the second half of the proposition.
\end{proof}

As observed above, Tucker core of size no more than $2$ is not uncommon.
Actually, Examples 2-5 provided in the next section all belong to this category.
We shall remark here that there are specific techniques available to
solve tensor problems {of size $2$}. For example, computing the Z-eigenvalues of a size $2$ tensor is equivalent to finding the common roots of two bivariate polynomials, and a numerical procedure  for solving the latter problem was discussed in~\cite{SorberVanBarelDeLathauwe2014}.




Now let us turn to the issue of estimating the error projected on the original tensor while working with a Tucker core approximately. Obviously, errors may occur in the process of Tucker decomposition; so the core tensor that we deal with may not be the true core tensor. The question is: Will the errors expand very quickly? We shall discuss the case for the CP decomposition here. The answer is negative.

\begin{proposition}
For a given tensor $\tensorT\in\complN N$, its independent (but not necessarily exact) Tucker decomposition
$
\tensorT = 
\llbracket \tensorG; \bA^{(1)}, \ldots, \bA^{(N)} \rrbracket
$ has the error
$$
\mbox{Err}_1=\|\tensorT-\tensorG\times_1\bA^{(1)}\dots \times_N\bA^{(N)} \|_F.
$$
Now we perform a CP decomposition on $\tensorG$ and get
$$
\widetilde{\tensorG} = \sum_{t=1} ^r \ba^{(1,t)} \circ \cdots \circ \ba^{(N,t)}
$$
with $\mbox{Err}_2=\|\tensorG-\widetilde{\tensorG} \|_F$.
Then
$$
\widetilde{\tensorT} = \widetilde{\tensorG}\times_1 \bA^{(1)}\times_2 \dots \times_N \bA^{(N)}
$$
is a CP decomposition of $\tensorT$ with an error estimation
$$
\|\tensorT- \widetilde{\tensorT} \|_F\leq \mbox{Err}_1+\mbox{Err}_2 \prod_{n=1}^N{\|\bA^{(n)}\|_2}.
$$
\end{proposition}
\noindent {\bf Proof:}
By Lemma~\ref{Lem:TensorProdMatrProd},
\begin{eqnarray*}
\widetilde{\tensorG}\times_1 \bA^{(1)}\times_2 \dots \times_N \bA^{(N)}&=& \sum_{t=1} ^r \ba^{(1,t)} \circ \cdots \circ \ba^{(N,t)}\times_1 \bA^{(1)}\times_2 \dots \times_N \bA^{(N)}\\
 &=& \sum_{t=1} ^r \bA^{(1)}\ba^{(1,t)} \circ \cdots \circ \bA^{(N)}\ba^{(N,t)},
\end{eqnarray*}
which is indeed a CP decomposition of $\tensorT$. Moreover, the error of this decomposition is:
\begin{eqnarray*}
&&\|\tensorT-\widetilde{\tensorT}\|_F = \|\tensorT-\widetilde{\tensorG}\times_1 \bA^{(1)}\times_2 \bA^{(2)}\dots \times_N \bA^{(N)}\|_F\\
&=&\|\tensorT-\tensorG\times_1 \bA^{(1)}\times_2 \bA^{(2)}\dots \times_N \bA^{(N)}+(\tensorG-\widetilde{\tensorG})\times_1 \bA^{(1)}\times_2 \bA^{(2)}\dots \times_N \bA^{(N)}\|_F\\
&\leq&\|\tensorT-\tensorG\times_1 \bA^{(1)}\times_2 \bA^{(2)}\dots \times_N \bA^{(N)}\|_F+\|(\tensorG-\widetilde{\tensorG})\times_1 \bA^{(1)}\times_2 \bA^{(2)}\dots \times_N \bA^{(N)}\|_F\\
&\le&\mbox{Err}_1+\mbox{Err}_2 \prod_{n=1}^N{\|\bA^{(n)}\|_2},
\end{eqnarray*}
where the last inequality is due to Lemma \ref{Lem:FrobeniusMatrixProduct}.
\hfill $\square$

We remark that when the Tucker decomposition is exact and orthonormal, then the above proposition reduces to Lemma 2.5 in~\cite{KhoromskijKhoromskaia2007}.
This proposition also suggests that one may indeed choose to work with a smaller Tucker core, and the resulting approximative CP decomposition will have a controllable error bound, thanks to an additive rate of the error accumulations.

\section{ Numerical experiments }


The goal of this section is to experiment if a reduced Tucker core helps to solve the tensor problem overall. The answer is, interestingly: {\it it depends}. If we apply the standard Alternating Least Squares (ALS) approach to find the CP decomposition, then \cite{TomasiBro2006} reported that more ALS iterations may be required on a  compressed core tensor. This is perhaps not very surprising, because in some cases the computational complexity does not necessarily go down with the size per se. Since we are not aware of a standard solver to compute the CP decomposition exactly to do the comparison, we choose to experiment with an `easier' computational object: the Z-eigenvalues and Z-eigenvectors of a symmetric tensor.

Most papers focussed on the computation of the largest or the smallest Z-eigenvalue; see~\cite{QiWangWang2009,NieWang2014,JiangMaZhang2015}. Most recently, Cui, Dai and Nie~\cite{AREoST} proposed an algorithm that can find every Z-eigenvalue of a given tensor.
Their idea is to formulate the problem of computing each Z-eigenvalue as a polynomial optimization problem, and then resort to the SOS method that in principle can globally solve any
polynomial optimization to optimality. In our numerical experiments, we used the method in~\cite{AREoST} to compute all the Z-eigenvalues and Z-eigenvectors for the original tensor and for its Tucker core tensor. In particular, we record both the running time of
Algorithm 3.6 in~\cite{AREoST} applied to some specific instances and the running time of the same algorithm on the corresponding Tucker core tensor plus the time consumed by the Tucker decomposition.
Our code is based on that in~\cite{AREoST} with some slight modifications and parameter tunings.

In the following, we choose five testing examples to do this experiment, and report the corresponding numerical results.
All of our experiments are run using MATLAB R2013a on a MacBook with an Intel dual core CPU at 1.3 GHz $\times 2$ and 4 GB of RAM, under an OS X 10.9.5 operating system.

%

\begin{example}
(Example 2 in \cite{OtZotSLT})

Consider the symmetric tensor $\tensorX\in \real{5\times 5\times 5\times 5}$ such that
$$
\tensorX(\bx^{\circ4})=(x_1+x_2+x_3+x_4)^4+(x_2+x_3+x_4+x_5)^4.
$$
The dimension of the original tensor is $5$ while the dimension of its core is $2$. It took 30.03 seconds to compute directly on the tensor itself while the computation on its core (including the Tucker decomposition) took only 3.29 seconds. The computed eigenvalues and eigenvectors of the original tensor are in Table \ref{OriginalExample1} and that of the core tensor are in Table \ref{CoreExample1}.
\begin{table}[htdp]
\begin{center}
\begin{tabular}{|c|c|c|c|c|c|}
\hline
Eigenvalues&\multicolumn{5}{|c|}{Eigenvectors}\\
\hline
24.500&-0.267&-0.535&-0.535&-0.535&-0.267\\
\hline
0.500&0.707&-0.000&-0.000&-0.000&-0.707\\
\hline
0.000&-0.073&0.618&-0.751&0.211&-0.067\\
\hline
\end{tabular}
\end{center}
\caption{Eigenvalues and Eigenvectors of the Original Tensor in Example 1}
\label{OriginalExample1}
\end{table}%

\begin{table}[htdp]
\begin{center}
\begin{tabular}{|c|c|c|}
\hline
Eigenvalues&\multicolumn{2}{|c|}{Eigenvectors}\\
\hline
24.500&-1.000&0.000\\
\hline
0.500&-0.000&-1.000\\
\hline
0.000&-0.354&-0.935\\
\hline
\end{tabular}
\end{center}
\caption{Eigenvalues and Eigenvectors of the Core Tensor in Example 1}
\label{CoreExample1}
\end{table}%

\end{example}

\begin{example}
(Example 3.5 in \cite{NieWang2014})

Consider the symmetric tensor $\tensorX\in \real{n\times n\times n}$ such that
$$
\tensorX_{ijk}=\frac{(-1)^i}{i}+\frac{(-1)^j}{j}+\frac{(-1)^k}{k} \ (1\leq i, j, k\leq n).
$$
For the case $n=8$, the dimension of the original tensor is $8$ while the dimension of its core is $2$.
The direct computation on the original tensor took 17670.71 seconds, and the computation on its core took 2.66 seconds.
The resulting eigenvalues and eigenvectors of the original tensor and its core tensor are in Table \ref{OriginalExample3} and Table \ref{CoreExample3} respectively.
\begin{table}[htdp]
\begin{center}
\begin{tabular}{|c|c|c|c|c|c|c|c|c|}
\hline
Eigenvalues&\multicolumn{8}{|c|}{Eigenvectors}\\

  \hline
 14.436&-0.687&-0.066&-0.411&-0.169&-0.356&-0.204&-0.332&-0.221\\
 \hline
8.586&-0.225& 0.579& 0.132& 0.445& 0.203& 0.400& 0.234& 0.378\\
\hline
0.000&-0.335& 0.283& 0.252& 0.256&-0.030&-0.751&-0.013& 0.337\\
\hline
 -14.436&-0.687&-0.066&-0.411&-0.169&-0.356&-0.204&-0.332&-0.221\\
 \hline
  -8.586&-0.225& 0.579& 0.132& 0.445& 0.203& 0.400& 0.234& 0.378\\
\hline
\end{tabular}
\end{center}
\caption{Eigenvalues and Eigenvectors of the Original Tensor in Example 2}
\label{OriginalExample3}
\end{table}%

\begin{table}[htdp]
\begin{center}
\begin{tabular}{|c|c|c|}
\hline
Eigenvalues&\multicolumn{2}{|c|}{Eigenvectors}\\
\hline
  14.436&-0.985&-0.175\\
  \hline
8.586& 0.488&-0.873\\
\hline
  -0.000& 0.344& 0.939\\
  \hline
 -14.436&-0.985&-0.175\\
 \hline
  -8.586& 0.488&-0.873\\
\hline
\end{tabular}
\end{center}
\caption{Eigenvalues and Eigenvectors of the Core Tensor in Example 2}
\label{CoreExample3}
\end{table}%

\end{example}

\begin{example}
(Example 4.14 in \cite{AREoST})

Consider the symmetric tensor $\tensorX\in \real{n\times n\times n\times n\times n}$:
$$
\tensorX_{i_1i_2i_3i_4i_5}=\ln(i_1)+\ln(i_2)+\ln(i_3)+\ln(i_4)+\ln(i_5)  \ (1\leq i_1, i_2, i_3, i_4, i_5\leq n).
$$
For the case $n=4$, the dimension of the original tensor is $4$ while the dimension of its core is $2$.
The direct computation on the original tensor took 186.58 seconds, and the computation on its core took 5.23 seconds.
The eigenvalues and eigenvectors of the original tensor and its core tensor are in Table \ref{OriginalExample6} and Table \ref{CoreExample6} respectively.

\begin{table}[htdp]
\begin{center}
\begin{tabular}{|c|c|c|c|c|}
\hline
Eigenvalues&\multicolumn{4}{|c|}{Eigenvectors}\\
\hline
 132.307&  0.403& 0.484& 0.532& 0.566\\
 \hline
0.707& -0.905&-0.308& 0.041& 0.289\\
 \hline
0.001&  0.565& 0.254&-0.022&-0.785\\
 \hline
-132.307&  0.403& 0.484& 0.532& 0.566\\
 \hline
  -0.707& -0.905&-0.308& 0.041& 0.289\\
\hline
\end{tabular}
\end{center}
\caption{Eigenvalues and Eigenvectors of the Original Tensor in Example 3}
\label{OriginalExample6}
\end{table}%

\begin{table}[htdp]
\begin{center}
\begin{tabular}{|c|c|c|}
\hline
Eigenvalues&\multicolumn{2}{|c|}{Eigenvectors}\\
\hline
 132.307&  1.000&-0.000\\
 \hline
0.707& -0.329&-0.944\\
 \hline
  -0.000&  0.127&-0.992\\
 \hline
-132.307&  1.000&-0.000\\
 \hline
  -0.707& -0.329&-0.944\\
\hline
\end{tabular}
\end{center}
\caption{Eigenvalues and Eigenvectors of the Core Tensor in Example 3}
\label{CoreExample6}
\end{table}%

\end{example}

{
\begin{example}
(Example 4.12 in \cite{AREoST})
Consider the symmetric tensor $\tensorX\in \real{n\times n\times n\times n}$:
$$
\tensorX_{i_1i_2i_3i_4}=\sin(i_1+i_2+i_3+i_4)\ (1\leq i_1, i_2, i_3, i_4\leq n).
$$
For the case $n=4$, the dimension of the original tensor is $4$ while the dimension of its core is $2$.
The direct computation on the original tensor took 74.75 seconds, and the computation on its core took 5.48 seconds.
The eigenvalues and eigenvectors of the original tensor and its core tensor are in Table \ref{OriginalExample4} and Table \ref{CoreExample4} respectively.

\begin{table}[htdp]
\begin{center}
\begin{tabular}{|c|c|c|c|c|}
\hline
Eigenvalues&\multicolumn{4}{|c|}{Eigenvectors}\\
\hline
4.632&  0.500&-0.133&-0.644&-0.563\\
\hline
2.991& -0.347&-0.766&-0.482& 0.246\\
\hline
0.000&  0.623&-0.587& 0.512& 0.068\\
\hline
  -5.645& -0.629&-0.485& 0.105& 0.598\\
\hline
  -2.525&  0.083&-0.621&-0.755&-0.194\\
\hline
\end{tabular}
\end{center}
\caption{Eigenvalues and Eigenvectors of the Original Tensor in Example 4}
\label{OriginalExample4}
\end{table}%

\begin{table}[htdp]
\begin{center}
\begin{tabular}{|c|c|c|}
\hline
Eigenvalues&\multicolumn{2}{|c|}{Eigenvectors}\\
\hline
  4.632&  -0.821& 0.571\\
\hline
  2.991&  -0.485&-0.874\\
\hline
  -5.645&  0.195&-0.981\\
\hline
  -2.525&  -0.969&-0.247\\
\hline
\end{tabular}
\end{center}
\caption{Eigenvalues and Eigenvectors of the Core Tensor in Example 4}
\label{CoreExample4}
\end{table}%

\end{example}
}

To conclude, the computational time spent on finding the Z-eigenvalues and Z-eigenvectors of a tensor can be substantially reduced if we turn to its Tucker core tensor instead. Table \ref{Tab: The Time Comparison between Algorithm and Revised Algorithm} summarizes the recorded computational times for the above examples 1 -- 5.

{
{\bf Acknowledgements.} We would like to thank Chunfeng Cui for sharing with us the codes on computing all Z-eigenvalues, and we thank Shmuel Friedland, Lek-Heng Lim, Jiawang Nie and Nikos Sidiropoulos for the fruitful discussions on the topics related to this paper.}

\begin{table}[htdp]

\begin{center}
\begin{tabular}{||c||c||c|c||c|c||}
\hline
{\it Example} & {\it Tensor Order } & {\it Tensor Size} & {\it Core Size} & {\it CPU for Tensor} & {\it CPU for Core} \\ \hline  \hline
1 & 4 & 5 & 2 & 30.03s & 3.29s \\
\hline
2&3&8&2&17670.70s&2.66s\\
\hline
3&5&4&2&186.58s&5.23s\\
\hline
{4}&{4}&{4}&{2}&{74.75s}&{5.48s}\\
\hline
\end{tabular}
\end{center}
\caption{Computational Time Comparison }
\label{Tab: The Time Comparison between Algorithm and Revised Algorithm}
\end{table}%
\bigskip

\bibliographystyle{unsrt}

\begin{thebibliography}{10}

\bibitem{BaderKolda2006}
B.~W. Bader and T.~G. Kolda.
\newblock Algorithm 862: Matlab tensor classes for fast algorithm prototyping.
\newblock {\em ACM Trans. Math. Software}, 32:635--653, 2006.

\bibitem{Bini2007}
D.~Bini.
\newblock {\em The Role of Tensor Rank in the Complexity Analysis of Bilinear
  Forms}.
\newblock ICIAM07, Z$\ddot{u}$rich, Switzerland, 2007.

\bibitem{Bloy2008}
L.~Bloy and R.~Verma.
\newblock On computing the underlying fiber directions from the diffusion
  orienta- tion distribution function.
\newblock In {\em International Conference on Medical Image Computing and
  Computer-assisted Intervention-part I}, pages 1--8, 2008.

\bibitem{BroAndersson1998}
R.~Bro and C.~A. Andersson.
\newblock Improving the speed of multi-way algorithms: Part ii.
\newblock {\em Compression, Chemometrics and Intelligent Laboratory Systems},
  42:105--113, 1998.

\bibitem{CARROLL}
J.~D. Carroll and J.~J. Chang.
\newblock Analysis of individual differences in multidimensional scaling via an
  n-way generalization of "eckart-young'' decomposition.
\newblock {\em Psychometrika}, 35:283--319, 1970.

\bibitem{CARROLL80}
J.~D. Carroll, S.~Pruzansky, and J.~B. Kruskal.
\newblock Candelinc: A general approach to multidimensional analysis of
  many-way arrays with linear constraints on parameters.
\newblock {\em Psychometrika}, 45:3--24, 1980.

\bibitem{STaSTR}
P.~Comon, G.~Golub, L.H. Lim, and B.~Mourrain.
\newblock Symmetric tensors and symmetric tensor rank.
\newblock {\em SIAM J. Matrix Anal. Appl.}, 30(3):1254--1279, 2008.

\bibitem{AREoST}
C.~Cui, Y.~Dai, and J.~Nie.
\newblock All real eigenvalues of symmetric tensors.
\newblock {\em SIAM J. Matrix Anal. Appl.}, 35:1582--1601, 2014.

\bibitem{FriedlandLim2014}
S.~Friedland and L.-H. Lim.
\newblock Computational complexity of tensor nuclear norm.
\newblock {\em Working Paper}, 2014.

\bibitem{Ghosh2008}
A.~Ghosh, E.~Tsigaridas, M.~Descoteaux, P.~Comon, B.~Mourrain, and R.~Deriche.
\newblock A polynomial based approach to extract the maxima of an antipodally
  symmetric spherical function and its application to extract fiber directions
  from the orientation distribution function in diffusion mri.
\newblock In {\em Computational Diffusion MRI Workshop (CDMRI'08), New York},
  2008.

\bibitem{Grothendieck1955}
A.~Gr\"othendieck.
\newblock Produits tensoriels topologiques et espaces nucl\'{e}aires.
\newblock {\em Mem. Amer. Math. Soc}, 16:140 pages, 1955.

\bibitem{Harshman}
R.~A. Harshman.
\newblock Foundations of the parafac procedure: Models and conditions for an
  "explanatory" multi-modal factor analysis.
\newblock {\em UCLA Working Paper}, 1969.

\bibitem{HillarLim2013}
C.J. Hillar and L.-H. Lim.
\newblock Most tensor problems are {NP}-hard.
\newblock {\em J. ACM}, 60(6):Art. 45, 39 pages, 2013.

\bibitem{TGMoMEatSVoaH}
J.~J. Hilling and A.~Sudbery.
\newblock The geometric measure of multipartite entanglement and the singular
  values of a hypermatrix.
\newblock {\em J. Math. Phys.}, 51(7):165--169, 2010.

\bibitem{ShenglongHu}
S.~Hu.
\newblock Relations of the nuclear norms of a tensor and its matrix
  flattenings.
\newblock {\em Linear Algebra and its Applications}, 478:188--199, 2015.

\bibitem{JiSzeZhouSoYe2013}
S.~Ji, K.-F. Sze, Z.~Zhou, A.~M.-C. So, and Y.~Ye.
\newblock Beyond convex relaxation: A polynomial-time non-convex optimization
  approach to network localization.
\newblock In {\em Proceedings of the 32nd IEEE International Conference on
  Computer Communications (INFOCOM 2013)}, 2013.

\bibitem{JiangLiZhang2015}
B.~Jiang, Z.~Li, and S.~Zhang.
\newblock On cones of nonnegative quartic forms.
\newblock {\em Found. Comput. Math. to apear}, 2015.

\bibitem{JiangMaZhang2015}
B.~Jiang, S.~Ma, and S.~Zhang.
\newblock Tensor principal component analysis via convex optimization.
\newblock {\em Math. Program.}, 150:423--457, 2015.

\bibitem{Khoromskaia2010}
V.~Khoromskaia.
\newblock {\em Numerical Solution of the Hartree-Fock Equation by Multilevel
  Tensor-structured Methods}.
\newblock PhD thesis, TU Berlin, 2010.

\bibitem{Khoromskij2006}
B.~N. Khoromskij.
\newblock Structured rank-$(r_1, \cdots, r_d)$ decomposition of
  function-related tensors in $\mathbb{R}^d$.
\newblock {\em Comp. Meth. in Appl. Math.}, 6:194--220, 2006.

\bibitem{Khoromskij2015}
B.~N. Khoromskij.
\newblock Tensor numerical methods for high-dimensional PDEs: Basic theory and
  initial applications. ESAIM: Proceedings and Surveys,
 48:1--28, 2015.

\bibitem{KhoromskijKhoromskaia2007}
B.~N. Khoromskij and V.~Khoromskaia.
\newblock Low rank Tucker-type tensor approximation to classical potentials.
\newblock {\em Central European J. Math.}, 5:523--550, 2007.

\bibitem{TDaA}
T.~G. Kolda and B.~W. Bader.
\newblock Tensor decompositions and applications.
\newblock {\em SIAM Review}, 51(3):455--500, 2009.

\bibitem{MOfHOD}
T.~G. Kolda and T.~Gibson.
\newblock Multilinear operators for higher-order decompositions.
\newblock Technical report, Sandia National Laboratories, 2006.

\bibitem{TGaA}
J.~M. Landsberg.
\newblock Tensors: Geometry and applications.
\newblock {\em Graduate Studies in Mathematics}, 2012.

\bibitem{DeLathauwerMoorVandewalle2000}
L.~De Lathauwer, B.~De Moor, and J.~Vandewalle.
\newblock A multilinear singular value decomposition.
\newblock {\em SIAM J. Matrix Anal. Appl.}, pages 1253--1278.

\bibitem{DeLathauwerMoorVandewalle2000b}
L.~De Lathauwer, B.~De Moor, and J.~Vandewalle.
\newblock On the best rank-1 and rank-($r_1,r_2,\cdots,r_n$) approximation of
  higher-order tensors.
\newblock {\em SIAM J. Matrix Anal. Appl.}, pages 1324--1342.

\bibitem{SVaEoTAVA}
L.-H. Lim.
\newblock Singular values and eigenvalues of tensors: A variational approach.
\newblock In {\em Computational Advances in Multi-Sensor Adaptive Processing,
  2005 1st IEEE International Workshop on}, pages 129--132. IEEE, 2005.

\bibitem{Nie2015}
J.~Nie.
\newblock Generating polynomials and symmetric tensor decompositions.
\newblock {\em Found. Comput. Math., to apear}, 2015.

\bibitem{NieWang2014}
J.~Nie and L.~Wang.
\newblock Semidefinite relaxations for best rank-1 tensor approximations.
\newblock {\em SIAM J. Matrix Anal. Appl.}, 35:1155--1179, 2014.

\bibitem{Qi2005}
L.~Qi.
\newblock Eigenvalues of a real supersymmetric tensor.
\newblock {\em J. Symbolic Comput.}, 40:1302--1324, 2005.

\bibitem{Qi2009}
L.~Qi, H.-H. Dai, and D.~Han.
\newblock Conditions for strong ellipticity and m-eigenvalues.
\newblock {\em Front. Math. China}, 4:349--364, 2009.

\bibitem{QiWangWang2009}
L.~Qi, F.~Wang, and Y.~Wang.
\newblock Z-eigenvalue methods for a global polynomial optimization problem.
\newblock {\em Math. Program.}, 118:301--316, 2009.

\bibitem{HOPSDTI}
L.~Qi, G.~Yu, and E.~X. Wu.
\newblock Higher order positive semi-definite diffusion tensor imaging.
\newblock {\em SIAM J. Imaging Sci.}, 3(3):416--433, 2010.

\bibitem{Schatten1950}
R.~Schatten.
\newblock {\em A Theory of Cross-Spaces}.
\newblock Princeton University Press, NJ, 1950.

\bibitem{SorberVanBarelDeLathauwe2014}
L.~Sorber, M.~Van Barel, and L.~De Lathauwer.
\newblock Numerical solution of bivariate and polyanalytic polynomial systems.
\newblock {\em SIAM J. Numer. Anal.}, pages 1551--1572.

\bibitem{TomasiBro2006}
G.~Tomasi and R.~Bro.
\newblock A comparison of algorithms for fitting the parafac model.
\newblock {\em Comput. Statist. Data Anal.}, 50:1700--1734, 2006.

\bibitem{Tucker1966}
L.~R. Tucker.
\newblock Some mathematical notes on three-mode factor analysis.
\newblock {\em Psychometrika}, 31:279--311, 1966.

\bibitem{CRoMDUTRD}
H.~Wang and N.~Ahuja.
\newblock Compact representation of multidimensional data using tensor rank-one
  decomposition.
\newblock {\em International Conference on Pattern Recognition}, 1:44--47,
  2014.

\bibitem{OtZotSLT}
J.~Xie and A.~Chang.
\newblock On the z-eigenvalues of the signless laplacian tensor for an even
  uniform hypergraph.
\newblock {\em Numer. Linear Algebra Appl.}, 20(6):1030--1045, 2013.

\bibitem{YuanZhang2015}
M.~Yuan and C.-H. Zhang.
\newblock On tensor completion via nuclear norm minimization.
\newblock {\em Found. Comput. Math.}, to appear.

\end{thebibliography}

\end{document}